\documentclass[11pt]{article}

\usepackage[a4paper,margin=1in]{geometry}
\usepackage{amsmath,amssymb,amsthm,mathtools}
\usepackage{microtype}
\usepackage{enumitem}
\usepackage[hidelinks]{hyperref}

\newtheorem{theorem}{Theorem}[section]
\newtheorem{lemma}[theorem]{Lemma}
\newtheorem{corollary}[theorem]{Corollary}
\theoremstyle{definition}
\newtheorem{definition}[theorem]{Definition}
\theoremstyle{remark}
\newtheorem{remark}[theorem]{Remark}

\newcommand{\qbinom}[3][]{%
  \genfrac{[}{]}{0pt}{}{#2}{#3}_{#1}%
}

\title{The Partition Pairing Theorems\\\ I}
\author{George E. Andrews \and Manosij Ghosh Dastidar}

\date{}

\begin{document}

\maketitle

\begin{abstract}
The aim of this paper is to introduce pairing theory for partitions. We begin with two statistics on integer partitions, the \emph{pairing index} and the \emph{pairing rank}.  The pairing index is equidistributed with the number of parts, while a joint refinement identifies its two constituents with the numbers of even and odd parts.  We further introduce the \emph{pairing width} and prove that pairing index and pairing width are jointly equidistributed with the number of parts and the largest part.  The resulting finite Gaussian generating function has a cyclotomic factorization from which Kummer's famous carry theorem for binomial coefficients follows.  We also prove a mod-$5$ congruence for the excess of unpaired parts congruent to $1$ modulo $4$ over those congruent to $3$ modulo $4$ in the partitions of $5n+4$.  A signed specialization exhibits that the parity of the pairing rank is governed by self-conjugate partitions.  Motivated by this, we go on to introduce a second, diagrammatic pairing: after the two wings of the Durfee square are folded together, the unpaired cells break into connected \emph{diagonal blocks}.  These blocks may be reflected independently, giving a Boolean decomposition of the set of partitions with a unique representative having all successive ranks nonnegative.  
We then relate our theory to overpartitions and Frobenius representations, obtaining as a corollary a geometric realization of overpartitions in terms of partitions whose principal hooks are all even.  Finally, we study simply paired partitions of negative pairing rank, obtaining identities involving odd divisors and overpartitions, a parity theorem for pairing rank $-2$, and a Toeplitz determinant whose coefficientwise limit is an explicit infinite product related to MacMahon's product for plane partitions.
\end{abstract}



\section{Introduction}

Let $\lambda$ be an integer partition. Pair equal parts of $\lambda$ as far as possible.  Thus, for each part occurring with odd multiplicity, exactly one copy remains unpaired, while every part of even multiplicity is completely paired.

Let $r(\lambda)$ be the largest part that participates in at least one pair; if $\lambda$ has distinct parts, set $r(\lambda)=0$.  Write the unpaired parts in decreasing order as
\[
 u_1>u_2>\cdots>u_s,
\]
and define
\[
 s(\lambda)=u_1-u_2+u_3-u_4+\cdots.
\]
The \emph{pairing index} of $\lambda$ is
\[
 T(\lambda)=r(\lambda)+s(\lambda).
\]

The statistic $T(\lambda)$ combines two rather different pieces of information: the largest repeated part and an alternating sum formed from the parts of odd multiplicity. Our first theorem shows that the pairing index is equidistributed with the length $\ell(\lambda)$, the number of parts of $\lambda$.

Equivalently, the theorem asserts the bivariate generating-function identity
\[
 \sum_{\lambda} z^{T(\lambda)}q^{|\lambda|}
 =\sum_{\lambda} z^{\ell(\lambda)}q^{|\lambda|}
 =\frac{1}{(zq;q)_\infty}.
\]
A concise generating-function proof of this identity is given below. 

Keeping the two constituents of $T(\lambda)$ separate yields a stronger joint theorem: the pair
\[
 \bigl(r(\lambda),s(\lambda)\bigr)
\]
is equidistributed with the pair consisting of the numbers of even and odd parts.  This joint identity produces a natural bivariate generating function that recurs throughout the paper.

We next introduce a second statistic, the \emph{pairing width} $W(\lambda)$, obtained by comparing the number of parts belonging to pairs with a complementary span determined by the unpaired parts.  The pair
\[
 \bigl(T(\lambda),W(\lambda)\bigr)
\]
is equidistributed with
\[
 \bigl(\ell(\mu),\mu_1\bigr).
\]
More precisely, if $N_{a,b}(n)$ counts partitions of $n$ with pairing index $a$ and pairing width $b$, then
\[
 \sum_{n\geq0}N_{a,b}(n)q^n
 =q^{a+b-1}\qbinom[q]{a+b-2}{a-1}.
\]
The cyclotomic factorization of this Gaussian polynomial gives a pairing-theoretic form of Kummer's theorem: for every prime $p$, the $p$-adic valuation of the total number, over all sizes, of partitions with fixed pairing index $a$ and width $b$ is the number of carries when $a-1$ and $b-1$ are added in base $p$.

The unpaired parts themselves also exhibit an arithmetic cancellation.  If $c_1(\lambda)$ and $c_3(\lambda)$ denote the numbers of unpaired parts of $\lambda$ congruent to $1$ and $3$ modulo $4$, respectively, then we prove
\[
 \sum_{\lambda\vdash 5n+4}
 \bigl(c_1(\lambda)-c_3(\lambda)\bigr)\equiv0\pmod5.
\]

Define the pairing rank by
\[
 \rho(\lambda)=r(\lambda)-s(\lambda).
\]
If $E(n)$ and $O(n)$ count the partitions of $n$ having even and odd pairing rank, respectively, then their difference is $(-1)^n$ times the number of self-conjugate partitions of $n$.

The appearance of self-conjugate partitions suggests that the original pairing operation has a geometric counterpart.  Our construction is naturally expressed in terms of successive ranks, introduced by Atkin~\cite{Atkin1966} and subsequently used in the partition-sieve work of Andrews and Bressoud~\cite{Andrews1972,Bressoud1980}; for a later study of successive-rank parity blocks, see Seo and Yee~\cite{SeoYee2018}.  We introduce a second pairing directly on the Young diagram.  Fold the two wings of the Durfee square across the main diagonal and pair the cells that coincide.  The cells left unpaired occur in connected \emph{diagonal blocks}.  These blocks can be reflected independently across the diagonal.  Consequently, the partitions of a fixed integer split into Boolean classes, each having a unique representative whose successive ranks are all nonnegative.  If $\kappa(\lambda)$ is the number of diagonal blocks, then for every $m\geq0$ we prove
\[
 \sum_{\substack{\lambda\vdash n\\ \text{all successive ranks}\geq0}}
 \binom{\kappa(\lambda)}{m}
 =p(n-2m^2)-p\bigl(n-2(m+1)^2\bigr).
\]
The case $\kappa(\lambda)=0$ consists exactly of the self-conjugate partitions, so the new diagrammatic pairing is tied directly to the pairing-rank parity theorem rather than being a separate construction.

Dividing the unsigned and signed specializations of the joint distribution yields precisely the bivariate generating function for overpartitions, in the standard sense of Corteel and Lovejoy~\cite{CorteelLovejoy2004}.  We show that this quotient also has a natural Young-diagram realization: overpartitions of $m$ are equinumerous with partitions of $2m$ for which every principal hook has even length.  This equinumerosity is refined by Durfee size, and combining it with the pairing-rank parity theorem gives a convolution linking ordinary partitions, self-conjugate partitions, and even principal hooks.

We next consider \emph{simply paired partitions}, in which no part occurs more than three times. This leads to two further identities for simply paired partitions of negative pairing rank: a signed enumeration by the number of paired parts yields the number of odd divisors, while the unsigned enumeration is half the number of overpartitions into odd parts.

Finally, let $A_2(n)$ denote the number of simply paired partitions of $n$ having pairing rank $-2$.  We prove a parity law for this function.  If $\chi(m)$ is $1$ when the odd part of $m$ is a square and is $0$ otherwise, then
\[
 A_2(n+8)+A_2(n)
 \equiv n+1+\chi(n+3)+\chi(n+7)\pmod 2.
\]
In particular, when $n$ is even, $A_2(n+8)$ and $A_2(n)$ have the same parity precisely when one of $n+3$ and $n+7$ is a perfect square.  This also produces four fixed parity progressions modulo $16$.

The full family of pairing-rank generating functions has an additional
determinantal structure.  If
\[
 R_m(q)=\sum_{\substack{\lambda\\\rho(\lambda)=m}}q^{|\lambda|},
\]
then we prove an exact finite Schur-function expansion for the Toeplitz
determinant formed from the functions $R_m(q)$.  As a consequence,
\[
 \lim_{k\to\infty}
 \det\bigl(R_{i-j}(q)\bigr)_{1\leq i,j\leq k}
 =\prod_{r\geq1}\frac{1}{(1-q^{2r+1})^r},
\]
where the limit is coefficientwise as a formal power series in $q$.  The product is reminiscent of MacMahon's classical generating function for plane partitions~\cite{MacMahon1916}; we make the comparison precise at the end of Section~8.

\section{The partition pairing theorem}

\begin{theorem}[First Partition Pairing Theorem]\label{thm:pairing}
For every $n\ge 0$ and $k\ge 0$, the number of partitions $\lambda$ of $n$ with pairing index $k$ 
equals the number of partitions of $n$ into exactly $k$ parts.
\end{theorem}

We first establish the finite generating function used in the proof.  We employ the standard notation
\[
 (a;q)_N=\prod_{j=0}^{N-1}(1-aq^j),
 \qquad
 (a;q)_\infty=\prod_{j\ge0}(1-aq^j),
\]
and
\[
 \qbinom[q]{A}{B}
 =
 \begin{cases}
 \dfrac{(q;q)_A}{(q;q)_B(q;q)_{A-B}},&0\le B\le A,\\[6pt]
 0,&\text{otherwise}.
 \end{cases}
\]
For standard partition and $q$-series notation, as well as the $q$-binomial theorem and the Jacobi triple product, we refer to~\cite[Chapters~1--2]{AndrewsPartitions}.

For $N\ge0$, we let $P_N(y,z,q)$ be the generating function for partitions whose parts are at most $N$, where $y$ marks $r(\lambda)$ and $z$ marks $s(\lambda)$:
\[
 P_N(y,z,q)
 =\sum_{\substack{\lambda\\\lambda_1\le N}}
 y^{r(\lambda)}z^{s(\lambda)}q^{|\lambda|}.
\]

\begin{lemma}\label{lem:finite}
For every $N\ge0$,
\begin{equation}\label{eq:finite-formula}
 P_N(y,z,q)
 =\sum_{m=0}^{N}\sum_{t=0}^{N}
 \frac{y^m z^t q^{2m+t}}{(q^2;q^2)_m}
 \qbinom[q^2]{N}{t}.
\end{equation}
\end{lemma}

\begin{proof}
We begin with a recurrence for $P_N$.  According to the multiplicity of the part $N$, there are four possibilities: $N$ does not occur; it occurs exactly once; it occurs a positive even number of times; or it occurs an odd number of times greater than one.  These four cases give
\begin{align}
 P_N(y,z,q)
 ={}&P_{N-1}(y,z,q)
 +z^Nq^N P_{N-1}(y,z^{-1},q) \notag\\
 &+\frac{q^{2N}}{1-q^{2N}}\,y^N P_{N-1}(1,z,q)
 +\frac{q^{3N}}{1-q^{2N}}\,y^Nz^N P_{N-1}(1,z^{-1},q).
 \label{eq:recurrence}
\end{align}
If $N$ occurs once, it becomes the largest unpaired part, so the alternating sum of the previous unpaired parts changes from $s$ to $N-s$; this accounts for the substitution $z\mapsto z^{-1}$ and the factor $z^N$.  If $N$ occurs at least twice, then $r(\lambda)=N$, so the previous value of $r$ is discarded; this explains the substitution $y=1$ in the last two terms.

We let $p_N(y,z,q)$ denote the right-hand side of \eqref{eq:finite-formula}.  Since
\[
 P_0(y,z,q)=p_0(y,z,q)=1,
\]
it suffices to show that $p_N$ satisfies \eqref{eq:recurrence}.

First,
\begin{align}
 &p_{N-1}(y,z,q)+z^Nq^Np_{N-1}(y,z^{-1},q)\notag\\
 &\quad=
 \sum_{m=0}^{N-1}\sum_{t=0}^{N}
 \frac{y^mz^tq^{2m+t}}{(q^2;q^2)_m}
 \left(
 \qbinom[q^2]{N-1}{t}
 +q^{2N-2t}\qbinom[q^2]{N-1}{t-1}
 \right)\notag\\
 &\quad=
 \sum_{m=0}^{N-1}\sum_{t=0}^{N}
 \frac{y^mz^tq^{2m+t}}{(q^2;q^2)_m}
 \qbinom[q^2]{N}{t},
 \label{eq:first-pair}
\end{align}
where the last step is the usual Gaussian-binomial recurrence
\[
 \qbinom[Q]{N}{t}
 =\qbinom[Q]{N-1}{t}
 +Q^{N-t}\qbinom[Q]{N-1}{t-1}.
\]
Thus \eqref{eq:first-pair} supplies every term in \eqref{eq:finite-formula} with $m<N$.

It remains to recover the terms with $m=N$.  The last two terms of \eqref{eq:recurrence} give
\begin{align*}
 &\frac{q^{2N}}{1-q^{2N}}\,y^Np_{N-1}(1,z,q)
 +\frac{q^{3N}}{1-q^{2N}}\,y^Nz^Np_{N-1}(1,z^{-1},q)\\
 &\quad=
 \frac{q^{2N}y^N}{1-q^{2N}}
 \sum_{t=0}^{N}z^tq^t\qbinom[q^2]{N}{t}
 \sum_{m=0}^{N-1}\frac{q^{2m}}{(q^2;q^2)_m}.
\end{align*}
The final sum telescopes, since for $m\ge1$,
\[
 \frac{q^{2m}}{(q^2;q^2)_m}
 =\frac{1}{(q^2;q^2)_m}-\frac{1}{(q^2;q^2)_{m-1}},
\]
and hence
\[
 \sum_{m=0}^{N-1}\frac{q^{2m}}{(q^2;q^2)_m}
 =\frac{1}{(q^2;q^2)_{N-1}}.
\]
Therefore
\begin{align*}
 &\frac{q^{2N}}{1-q^{2N}}\,y^Np_{N-1}(1,z,q)
 +\frac{q^{3N}}{1-q^{2N}}\,y^Nz^Np_{N-1}(1,z^{-1},q)\\
 &\qquad=
 \frac{y^Nq^{2N}}{(q^2;q^2)_N}
 \sum_{t=0}^{N}z^tq^t\qbinom[q^2]{N}{t},
\end{align*}
which is exactly the $m=N$ portion of \eqref{eq:finite-formula}.  Together with \eqref{eq:first-pair}, this proves the lemma.
\end{proof}

\begin{proof}[Proof of Theorem~\ref{thm:pairing}]
We set $y=z$ in Lemma~\ref{lem:finite} and let $N\to\infty$.  Since
\[
 \lim_{N\to\infty}\qbinom[q^2]{N}{t}
 =\frac{1}{(q^2;q^2)_t},
\]
we obtain
\begin{align*}
 \sum_{\lambda}z^{T(\lambda)}q^{|\lambda|}
 &=\lim_{N\to\infty}P_N(z,z,q)\\
 &=\sum_{m\ge0}\sum_{t\ge0}
 \frac{z^{m+t}q^{2m+t}}
 {(q^2;q^2)_m(q^2;q^2)_t}\\
 &=\frac{1}{(zq^2;q^2)_\infty}\,
   \frac{1}{(zq;q^2)_\infty}\\
 &=\frac{1}{(zq;q)_\infty}.
\end{align*}
But
\[
 \frac{1}{(zq;q)_\infty}
 =\prod_{j\ge1}\frac{1}{1-zq^j}
 =\sum_{\lambda}z^{\ell(\lambda)}q^{|\lambda|},
\]
the ordinary generating function for partitions in which the exponent of $z$ records the number of parts.  Comparing coefficients of $z^kq^n$ proves the theorem.
\end{proof}

The proof above uses only the specialization $y=z$.  Keeping the two
variables separate reveals that the two constituents of the pairing index
already have independent classical interpretations.  For a partition $\mu$,
we let $e(\mu)$ and $o(\mu)$ denote, respectively, the numbers of even and odd
parts of $\mu$, counted with multiplicity.

\begin{theorem}[Joint Pairing Theorem]\label{thm:joint-pairing}
For every $n,a,b\geq0$, the number of partitions $\lambda$ of $n$ satisfying
\[
 r(\lambda)=a,\qquad s(\lambda)=b
\]
equals the number of partitions with $a$ even parts and $b$ odd parts. 
\end{theorem}

\begin{proof}
Letting $N\to\infty$ in Lemma~\ref{lem:finite}, without identifying $y$
and $z$, gives
\begin{align*}
 \sum_{\lambda}y^{r(\lambda)}z^{s(\lambda)}q^{|\lambda|}
 &=\sum_{m,t\geq0}
   \frac{y^mz^tq^{2m+t}}
   {(q^2;q^2)_m(q^2;q^2)_t}\\
 &=\frac{1}{(yq^2;q^2)_\infty(zq;q^2)_\infty}.
\end{align*}
On the other hand,
\[
 \frac{1}{(yq^2;q^2)_\infty(zq;q^2)_\infty}
 =\prod_{j\geq1}\frac{1}{1-yq^{2j}}
  \prod_{j\geq1}\frac{1}{1-zq^{2j-1}}.
\]
In the first product the exponent of $y$ records the number of even parts,
while in the second the exponent of $z$ records the number of odd parts.
Thus the product is
\[
 \sum_{\mu}y^{e(\mu)}z^{o(\mu)}q^{|\mu|}.
\]
Comparing coefficients of $y^az^bq^n$ proves the theorem.
\end{proof}

\begin{remark}[Classical specializations]\label{rem:joint-classical}
Theorem~\ref{thm:joint-pairing} specializes to several classical partition theorems.  Summing over $b$ gives the Andrews--Deutsch theorem: the largest repeated part is equidistributed with the number of even parts~\cite{AndrewsDeutsch2016}.  Taking $a=0$ gives Bessenrodt's refinement of Euler's theorem: distinct partitions with alternating sum $b$ are equinumerous with partitions into exactly $b$ odd parts~\cite{Bessenrodt1994}; summing over $b$ recovers Euler's odd--distinct theorem.  Finally, within the specialization $a=0$, setting $z=q$ in the generating function and then replacing $q^2$ by $q$ yields Schmidt's theorem, which states that $p(n)$ counts distinct partitions whose odd-indexed parts sum to $n$~\cite{AndrewsPaule2022}.  Thus the Joint Pairing Theorem provides a single two-variable framework for these classical identities.
\end{remark}

We refine Theorem~\ref{thm:pairing} by recording a second statistic.  The
\emph{pair span} of $\lambda$ is the number of parts of $\lambda$ that belong
to pairs.  If the unpaired parts are
\[
 u_1>u_2>\cdots>u_j
\]
and $j>0$, we define the \emph{residual span} by
\[
 2\bigl(u_1-s(\lambda)\bigr)+1;
\]
when there are no unpaired parts, we define the residual span to be $0$.  The
\emph{pairing width} $W(\lambda)$ is the larger of the pair span and the
residual span.

The following finite generating function handles the unpaired
portion.

\begin{lemma}[Residual-span generating function]\label{lem:residual-span}
For $h\geq0$,
\[
 \sum_{\substack{\delta\text{ distinct}\\
                  \delta=\varnothing\ \text{or}\ 
                  \delta_1-s(\delta)\leq h}}
 x^{s(\delta)}q^{|\delta|}
 =\frac{1}{(xq;q^2)_{h+1}}.
\]
\end{lemma}

\begin{proof}
We fix $t\geq1$ and write a distinct partition as
\[
 \delta=(u_1>u_2>\cdots>u_\ell>0).
\]
We introduce the nonnegative gaps
\[
 g_i=u_i-u_{i+1}-1\quad(1\leq i<\ell),
 \qquad g_\ell=u_\ell-1.
\]
Then
\begin{align}
 |\delta|&=\frac{\ell(\ell+1)}2+\sum_{i=1}^{\ell}i g_i,\label{eq:gap-weight}\\
 s(\delta)&=\left\lceil\frac\ell2\right\rceil
             +\sum_{\substack{1\leq i\leq\ell\\ i\ \mathrm{odd}}}g_i,\label{eq:gap-alt}\\
 u_1-s(\delta)&=\left\lfloor\frac\ell2\right\rfloor
             +\sum_{\substack{1\leq i\leq\ell\\ i\ \mathrm{even}}}g_i.\label{eq:gap-res}
\end{align}
We put $Q=q^2$.  For fixed $s(\delta)=t$, the contributions from lengths
$\ell=2j$ and $\ell=2j+1$ are respectively
\[
 q^tQ^{j^2}
 \qbinom[Q]{t-1}{j-1}\qbinom[Q]{h}{j}
\]
and
\[
 q^tQ^{j(j+1)}
 \qbinom[Q]{t-1}{j}\qbinom[Q]{h}{j}.
\]
These formulas follow directly from
\eqref{eq:gap-weight}--\eqref{eq:gap-res}: the odd-indexed gaps have a
prescribed total, while the even-indexed gaps have total at most the indicated
bound.  Adding the two contributions and using
\[
 \qbinom[Q]{t-1}{j-1}+Q^j\qbinom[Q]{t-1}{j}
 =\qbinom[Q]{t}{j}
\]
gives
\[
 [x^t]\sum_{\substack{\delta\text{ distinct}\\
                  \delta_1-s(\delta)\leq h}}
 x^{s(\delta)}q^{|\delta|}
 =q^t\sum_{j\geq0}Q^{j^2}
   \qbinom[Q]{t}{j}\qbinom[Q]{h}{j}.
\]
The finite $q$-Vandermonde identity yields
\[
 \sum_{j\geq0}Q^{j^2}
   \qbinom[Q]{t}{j}\qbinom[Q]{h}{j}
 =\qbinom[Q]{t+h}{t}.
\]
The empty partition supplies the case $t=0$.  Hence
\[
 \sum_{\substack{\delta\text{ distinct}\\
                  \delta=\varnothing\ \text{or}\ 
                  \delta_1-s(\delta)\leq h}}
 x^{s(\delta)}q^{|\delta|}
 =\sum_{t\geq0}(xq)^t\qbinom[q^2]{t+h}{t}.
\]
The $q$-binomial theorem now gives
\[
 \sum_{t\geq0}(xq)^t\qbinom[q^2]{t+h}{t}
 =\frac{1}{(xq;q^2)_{h+1}},
\]
as required.
\end{proof}

\begin{theorem}[Pairing index--width theorem]\label{thm:index-width}
For every $n\geq0$ and $a,b\geq0$, the number of partitions of $n$ having
pairing index $a$ and pairing width $b$ equals the number of partitions
of $n$ having exactly $a$ parts and largest part $b$.
\end{theorem}

\begin{proof}
For $B\geq0$, we put
\[
 F_B(x,q)=\sum_{\substack{\lambda\\W(\lambda)\leq B}}
 x^{T(\lambda)}q^{|\lambda|}.
\]
We prove
\begin{equation}\label{eq:width-bounded}
 F_B(x,q)=\frac{1}{(xq;q)_B}.
\end{equation}
The case $B=0$ is immediate.  We assume $B\geq1$.

After maximal pairing, we write $\lambda$ uniquely as a paired portion $\pi$
and a distinct unpaired portion $\delta$: for each extracted pair $j+j$,
we place one copy of $j$ in $\pi$, while $\delta$ consists of the parts left
unpaired.  Then
\[
 |\lambda|=2|\pi|+|\delta|,
 \qquad
 T(\lambda)=\pi_1+s(\delta),
\]
with $\pi_1=0$ when $\pi$ is empty.  Moreover, the pair span is
$2\ell(\pi)$.

We set
\[
 m=\left\lfloor\frac B2\right\rfloor,
 \qquad
 h=\left\lfloor\frac{B-1}{2}\right\rfloor.
\]
The condition $W(\lambda)\leq B$ is equivalent to the two independent
conditions
\[
 \ell(\pi)\leq m,
 \qquad
 \delta=\varnothing\ \text{or}\ 
 \delta_1-s(\delta)\leq h.
\]
Consequently $F_B$ factors into a paired and an unpaired contribution.
For the paired contribution, conjugating $\pi$ gives
\[
 \sum_{\ell(\pi)\leq m}x^{\pi_1}q^{2|\pi|}
 =\frac{1}{(xq^2;q^2)_m}.
\]
By Lemma~\ref{lem:residual-span}, the unpaired contribution is
\[
 \frac{1}{(xq;q^2)_{h+1}}.
\]
Hence
\[
 F_B(x,q)
 =\frac{1}{(xq^2;q^2)_{\lfloor B/2\rfloor}
            (xq;q^2)_{\lceil B/2\rceil}}
 =\frac{1}{(xq;q)_B},
\]
which proves \eqref{eq:width-bounded}.

The right-hand side is the generating function for partitions whose largest
part is at most $B$, with $x$ marking the number of parts.  Thus
\[
 \#\{\lambda\vdash n:T(\lambda)=a,\ W(\lambda)\leq B\}
 =\#\{\mu\vdash n:\ell(\mu)=a,\ \mu_1\leq B\}.
\]
Subtracting the corresponding identity with $B-1$ in place of $B$ gives the
assertion with pairing width and largest part both equal to $B$.
\end{proof}

The proof also yields the following exact finite polynomial, which we record for later use.
If
\[
 N_{a,b}(n)=\#\{\lambda\vdash n:T(\lambda)=a,\ W(\lambda)=b\},
\]
then
\begin{equation}\label{eq:width-gaussian}
 \sum_{n\geq0}N_{a,b}(n)q^n
 =q^{a+b-1}\qbinom[q]{a+b-2}{a-1}
 \qquad(a,b\geq1).
\end{equation}
Indeed,
\[
 F_b(x,q)-F_{b-1}(x,q)
 =\frac{xq^b}{(xq;q)_b},
\]
and coefficient extraction by the $q$-binomial theorem gives
\eqref{eq:width-gaussian}.

This Gaussian polynomial also yields a classical arithmetic theorem as an
immediate consequence.

\begin{corollary}[Pairing form of Kummer's theorem]\label{cor:kummer}
Let $p$ be prime and let $a,b\geq1$.  The $p$-adic valuation of the total
number, over all sizes, of partitions having pairing index $a$ and pairing
width $b$ equals the number of carries that occur when $a-1$ and $b-1$
are added in base $p$.
\end{corollary}

\begin{proof}
Setting $q=1$ in \eqref{eq:width-gaussian} gives
\begin{equation}\label{eq:width-total-binomial}
 \sum_{n\geq0}N_{a,b}(n)
 =\binom{a+b-2}{a-1}.
\end{equation}
The carry rule is already visible in the cyclotomic factorization of the Gaussian
polynomial.  We put
\[
 r=a-1,\qquad s=b-1.
\]
For $d\geq2$, the exponent of the cyclotomic polynomial $\Phi_d(q)$ in
$\qbinom[q]{r+s}{r}$ is
\begin{equation}\label{eq:cyclotomic-exponent}
 e_d=\left\lfloor\frac{r+s}{d}\right\rfloor
     -\left\lfloor\frac r d\right\rfloor
     -\left\lfloor\frac s d\right\rfloor.
\end{equation}
Writing
\[
 r=du+\rho,\qquad s=dv+\sigma,
 \qquad0\leq\rho,\sigma<d,
\]
reduces \eqref{eq:cyclotomic-exponent} to
\[
 e_d=\left\lfloor\frac{\rho+\sigma}{d}\right\rfloor.
\]
Thus $e_d$ is either $0$ or $1$, and
\[
 e_d=1
 \quad\Longleftrightarrow\quad
 (r\bmod d)+(s\bmod d)\geq d.
\]
Taking $d=p^j$ with $j\geq1$, this is exactly the condition that the addition of $r$ and
$s$ in base $p$ produces a carry into the $p^j$-place.

Finally,
\[
 \Phi_{p^j}(1)=p,
\]
whereas $\Phi_d(1)=1$ when $d>1$ is not a prime power.  Therefore the
$p$-adic valuation of the specialization
\[
 \qbinom[q]{r+s}{r}\Big|_{q=1}=\binom{r+s}{r}
\]
is the number of such carries.  Together with
\eqref{eq:width-total-binomial}, this gives the stated pairing formulation and
recovers Kummer's classical theorem~\cite{Kummer1852}.
\end{proof}

For a partition $\lambda$, we let $c_1(\lambda)$ and $c_3(\lambda)$ denote,
respectively, the numbers of unpaired parts congruent to $1$ and $3$ modulo
$4$, and we put
\[
 \delta(\lambda)=c_1(\lambda)-c_3(\lambda).
\]

\begin{theorem}\label{thm:unpaired-mod5}

Among all the unpaired parts in the partitions of $5n+4$, the
number of parts congruent to $1$ modulo $4$ minus the number congruent to $3$
modulo $4$ is divisible by $5$.
\end{theorem}

It is nice to note that the progression $5n+4$ is the one appearing in Ramanujan's classical congruence $p(5n+4)\equiv0\pmod5$~\cite{Ramanujan1921}. 
\begin{proof}
We consider the two-variable generating function
\[
 F(z,q)=\sum_{\lambda}z^{\delta(\lambda)}q^{|\lambda|}.
\]
Even part sizes do not affect the exponent of $z$.  For a part
of size $4j+1$, an odd multiplicity leaves one unpaired part and hence contributes a
factor $z$, whereas an even multiplicity leaves no unpaired part.  Its local
factor is therefore
\[
 1+zq^{4j+1}+q^{2(4j+1)}+zq^{3(4j+1)}+\cdots
 =\frac{1+zq^{4j+1}}{1-q^{2(4j+1)}}.
\]
Similarly, the local factor for a part of size $4j+3$ is
\[
 \frac{1+z^{-1}q^{4j+3}}{1-q^{2(4j+3)}}.
\]
Consequently,
\begin{align*}
 F(z,q)
 &=\frac{(-zq;q^4)_\infty(-z^{-1}q^3;q^4)_\infty}
 {(q^2;q^2)_\infty(q^2;q^4)_\infty}\\
 &=\frac{(q^4;q^4)_\infty(-zq;q^4)_\infty
 (-z^{-1}q^3;q^4)_\infty}{(q^2;q^2)_\infty^2}.
\end{align*}
By the Jacobi triple product~\cite[Chapter~2]{AndrewsPartitions},
\[
 (q^4;q^4)_\infty(-zq;q^4)_\infty(-z^{-1}q^3;q^4)_\infty
 =\sum_{r=-\infty}^{\infty}z^r q^{2r^2-r},
\]
and hence
\[
 F(z,q)=\frac{1}{(q^2;q^2)_\infty^2}
 \sum_{r=-\infty}^{\infty}z^r q^{2r^2-r}.
\]
We use Jacobi's identity~\cite[Chapter~2]{AndrewsPartitions}
\[
 (q^2;q^2)_\infty^3
 =\sum_{m=0}^{\infty}(-1)^m(2m+1)q^{m^2+m}.
\]
Multiplying numerator and denominator by $(q^2;q^2)_\infty^3$ gives
\[
 F(z,q)=\frac{1}{(q^2;q^2)_\infty^5}
 \left(\sum_{r=-\infty}^{\infty}z^r q^{2r^2-r}\right)
 \left(\sum_{m=0}^{\infty}(-1)^m(2m+1)q^{m^2+m}\right).
\]
Since $(1-x)^5\equiv1-x^5\pmod5$, we have
\[
 (q^2;q^2)_\infty^5\equiv(q^{10};q^{10})_\infty\pmod5.
\]
Therefore, coefficientwise modulo $5$,
\[
 F(z,q)\equiv\frac{1}{(q^{10};q^{10})_\infty}
 \left(\sum_{r=-\infty}^{\infty}z^r q^{2r^2-r}\right)
 \left(\sum_{m=0}^{\infty}(-1)^m(2m+1)q^{m^2+m}\right).
\]
The residues modulo $5$ are
\[
 2r^2-r\equiv0,1,1,0,3
 \qquad (r\equiv0,1,2,3,4\pmod5),
\]
and
\[
 m^2+m\equiv0,2,1,2,0
 \qquad (m\equiv0,1,2,3,4\pmod5).
\]
Thus the only way the exponent in the numerator can be congruent to $4$
modulo $5$ is to have
\[
 r\equiv4\pmod5,
 \qquad
 m\equiv2\pmod5.
\]
But then $2m+1\equiv0\pmod5$.  Since
$1/(q^{10};q^{10})_\infty$ contains only powers of $q$ divisible by $5$, it
follows that
\[
 [q^{5n+4}]F(z,q)\equiv0\pmod5
\]
as a Laurent polynomial in $z$.  Finally,
\[
 \sum_{\lambda\vdash 5n+4}\delta(\lambda)
 =\left.[q^{5n+4}]\,z\frac{\partial}{\partial z}F(z,q)\right|_{z=1},
\]
so the required congruence follows.
\end{proof}

\begin{remark}
The proof gives the stronger coefficientwise statement
\[
 \#\{\lambda\vdash5n+4:\delta(\lambda)=j\}\equiv0\pmod5
 \qquad(j\in\mathbb Z).
\]
\end{remark}

\section{Pairing-rank parity and self-conjugate partitions}

We let
\[
 \rho(\lambda)=r(\lambda)-s(\lambda)
\]
be the \emph{pairing rank} of $\lambda$.  For $n\geq0$, we let $E(n)$ and $O(n)$ denote the numbers of partitions of $n$ having even and odd pairing rank, respectively, and we let $\operatorname{sc}(n)$ denote the number of self-conjugate partitions of $n$.

\begin{theorem}\label{thm:rank-parity}
For every $n\geq0$,
\[
 E(n)-O(n)=(-1)^n\operatorname{sc}(n).
\]
\end{theorem}

\begin{proof}
By the definition of the pairing rank,
\[
 E(n)-O(n)=\sum_{\lambda\vdash n}(-1)^{\rho(\lambda)}.
\]
Because $r(\lambda)-s(\lambda)\equiv r(\lambda)+s(\lambda)\pmod2$,
\[
 (-1)^{\rho(\lambda)}
 =(-1)^{r(\lambda)-s(\lambda)}
 =(-1)^{r(\lambda)+s(\lambda)}.
\]
Thus the desired signed generating function is obtained from $P_N(y,z,q)$ by setting $y=z=-1$ and then letting $N\to\infty$.

Lemma~\ref{lem:finite}, together with
\[
 \lim_{N\to\infty}\qbinom[q^2]{N}{t}
 =\frac{1}{(q^2;q^2)_t},
\]
gives, for arbitrary $y$ and $z$,
\begin{align*}
 \sum_{\lambda}y^{r(\lambda)}z^{s(\lambda)}q^{|\lambda|}
 &=\sum_{m,t\geq0}
 \frac{y^mz^tq^{2m+t}}
 {(q^2;q^2)_m(q^2;q^2)_t}\\
 &=\frac{1}{(yq^2;q^2)_\infty(zq;q^2)_\infty}.
\end{align*}
Consequently,
\begin{align}
 \sum_{n\geq0}\bigl(E(n)-O(n)\bigr)q^n
 &=\sum_{\lambda}(-1)^{r(\lambda)+s(\lambda)}q^{|\lambda|}\notag\\
 &=\frac{1}{(-q^2;q^2)_\infty(-q;q^2)_\infty}\notag\\
 &=\frac{1}{(-q;q)_\infty}.\label{eq:rank-signed-gf}
\end{align}
Using $1+q^j=(1-q^{2j})/(1-q^j)$, we obtain
\[
 \frac{1}{(-q;q)_\infty}
 =\frac{(q;q)_\infty}{(q^2;q^2)_\infty}
 =\prod_{j\geq1}(1-q^{2j-1}).
\]

Self-conjugate partitions are equinumerous with partitions into distinct odd parts~\cite[Chapter~1]{AndrewsPartitions}, and therefore
\[
 \sum_{n\geq0}\operatorname{sc}(n)q^n
 =\prod_{j\geq1}(1+q^{2j-1}).
\]
Replacing $q$ by $-q$ gives
\[
 \sum_{n\geq0}(-1)^n\operatorname{sc}(n)q^n
 =\prod_{j\geq1}(1-q^{2j-1}).
\]
Together with \eqref{eq:rank-signed-gf}, this yields
\[
 \sum_{n\geq0}\bigl(E(n)-O(n)\bigr)q^n
 =\sum_{n\geq0}(-1)^n\operatorname{sc}(n)q^n.
\]
Comparing coefficients of $q^n$ proves the theorem.
\end{proof}

\section{Diagonal pairings}\label{sec:diagonal-pairing}

The preceding theorem brings self-conjugate partitions into the pairing theory in a rather direct way.  This suggests asking whether the original pairing procedure has a counterpart directly on the Young diagram.  The following construction provides one.  The original procedure pairs equal parts and records what remains unpaired; here we pair cells reflected across the main diagonal and again study the unpaired remainder.

We let $\lambda$ have Durfee size $d$, and write
\[
 \alpha_i=\lambda_i-d,
 \qquad
 \beta_i=\lambda_i'-d,
 \qquad 1\leq i\leq d.
\]
Thus $\alpha=(\alpha_1,\ldots,\alpha_d)$ is the wing to the right of the Durfee square, while $\beta=(\beta_1,\ldots,\beta_d)$ is the lower wing after reflection in the main diagonal.  Both are partitions with at most $d$ parts, and
\[
 |\lambda|=d^2+|\alpha|+|\beta|.
\]
We superimpose the diagrams of $\alpha$ and $\beta$.  A cell common to both diagrams represents a pair of cells of $\lambda$ reflected across the main diagonal; we call such a pair \emph{diagonally paired}.  We put
\[
 I_i=\min(\alpha_i,\beta_i),
 \qquad
 U_i=\max(\alpha_i,\beta_i).
\]
The cells left after all possible diagonal pairings form the skew diagram $U/I$.

\begin{definition}
The edge-connected components of $U/I$ are the \emph{diagonal blocks} of $\lambda$.  We write $\kappa(\lambda)$ for their number.
\end{definition}

Connected components are the natural units for this construction: a block is the smallest portion that may be moved from one side of the diagonal to the other without destroying the partition property.

\begin{lemma}[Block-flipping lemma]\label{lem:block-flipping}
Fix $d$, $I$, and $U$.  Every diagonal block of $U/I$ may be assigned independently to either wing.  Every such assignment gives a pair of partitions $\alpha,\beta$ satisfying
\[
 \alpha\cap\beta=I,
 \qquad
 \alpha\cup\beta=U,
\]
and every such pair arises in this way.
\end{lemma}

\begin{proof}
We regard Young diagrams as order ideals in the product order on positive integer pairs.  If two cells of $U/I$ share an edge, they cannot lie one in $\alpha\setminus\beta$ and the other in $\beta\setminus\alpha$: horizontally this follows because each row of a Young diagram is an initial segment, and vertically it follows because the row lengths are weakly decreasing.  Hence every connected component of $U/I$ lies wholly in one of the two differences.

Distinct connected components of a skew Young diagram are separated diagonally and are therefore pairwise incomparable in the product order.  Consequently, adjoining any chosen collection of components to $I$ again gives an order ideal: if a cell lies in a chosen component, then every smaller cell of $U$ lies either in $I$ or in that same component.  Thus each component may be placed independently in $\alpha$ or in $\beta$.  The construction is reversible.
\end{proof}

Following Atkin~\cite{Atkin1966}, for a partition of Durfee size $d$, we write its successive ranks as
\[
 \sigma_i(\lambda)=\lambda_i-\lambda_i'
 =\alpha_i-\beta_i,
 \qquad 1\leq i\leq d.
\]
We let $\mathcal P^+$ denote the set of partitions whose successive ranks are all nonnegative.

\begin{theorem}[Diagonal Pairing Theorem]\label{thm:diagonal-pairing}
The set of all partitions is a disjoint union of Boolean classes under diagonal-block flips.  A class with $k$ diagonal blocks contains exactly $2^k$ partitions, all of the same size, and contains a unique member of $\mathcal P^+$.

In other words,
\begin{equation}\label{eq:boolean-euler}
 \sum_{\lambda\in\mathcal P^+}2^{\kappa(\lambda)}q^{|\lambda|}
 =\frac{1}{(q;q)_\infty}.
\end{equation}
\end{theorem}

\begin{proof}
By Lemma~\ref{lem:block-flipping}, fixing $I$ and $U$ gives one independent binary choice for each diagonal block, so the corresponding Boolean class has cardinality $2^{\kappa(\lambda)}$.  The size is unchanged because a flip merely moves cells between the two wings.

Among these choices there is exactly one for which every block is placed in the upper wing.  For this member,
\[
 \alpha=U,
 \qquad
 \beta=I,
\]
so $\alpha_i\geq\beta_i$ for every $i$, or equivalently $\sigma_i(\lambda)\geq0$ for every successive rank.  Conversely, if all successive ranks are nonnegative, then necessarily $\alpha=U$ and $\beta=I$.  This proves uniqueness.  Summing the sizes of the Boolean classes gives \eqref{eq:boolean-euler}.
\end{proof}

For example, we take $\lambda=(4,3,1)$.  Its Durfee size is $2$, with
\[
 \alpha=(2,1),
 \qquad
 \beta=(1,0).
\]
The two cells of $U/I$ meet only at a corner, so $\kappa(\lambda)=2$.  Independently reflecting these two blocks gives the four partitions
\[
 (4,3,1),\qquad (4,2,2),\qquad (3,3,1,1),\qquad (3,2,2,1).
\]
The unique member with all successive ranks nonnegative is $(4,3,1)$.  Ordinary conjugation flips every block, and therefore acts as the antipodal map on each Boolean class.

The zero-block case links the construction directly to the preceding section.

\begin{corollary}\label{cor:zero-block-self-conjugate}
A partition has no diagonal blocks if and only if it is self-conjugate.  Consequently, Theorem~\ref{thm:rank-parity} may be written
\[
 E(n)-O(n)
 =(-1)^n\#\{\lambda\vdash n:\kappa(\lambda)=0\}.
\]
Thus the parity of the pairing rank detects exactly the partitions for which the diagonal pairing is complete.
\end{corollary}

\begin{proof}
We have $\kappa(\lambda)=0$ if and only if $U=I$, equivalently $\alpha=\beta$.  Since the Durfee square is itself symmetric, this is equivalent to $\lambda=\lambda'$.
\end{proof}

The Boolean decomposition also gives a second realization of Euler's partition generating function.  More is true: the entire distribution of the diagonal blocks has a simple form.

\begin{theorem}\label{thm:refined-diagonal-pairing}
For every $m\geq0$,
\begin{equation}\label{eq:block-moment-gf}
 \sum_{\lambda\in\mathcal P^+}
 \binom{\kappa(\lambda)}{m}q^{|\lambda|}
 =\frac{q^{2m^2}-q^{2(m+1)^2}}{(q;q)_\infty}.
\end{equation}
Equivalently, for every $n,m\geq0$,
\begin{equation}\label{eq:block-moment-coeff}
 \sum_{\substack{\lambda\vdash n\\ \lambda\in\mathcal P^+}}
 \binom{\kappa(\lambda)}{m}
 =p(n-2m^2)-p\bigl(n-2(m+1)^2\bigr),
\end{equation}
where $p(N)=0$ for $N<0$.
\end{theorem}

\begin{proof}
We let $\lambda\in\mathcal P^+$ have Frobenius symbol (in the standard notation; see, for example,~\cite[Chapter~1]{AndrewsPartitions})
\[
 \begin{pmatrix}
  a_1&a_2&\cdots&a_d\\
  b_1&b_2&\cdots&b_d
 \end{pmatrix},
 \qquad
 a_1>\cdots>a_d\geq0,
 \quad
 b_1>\cdots>b_d\geq0.
\]
The condition $\lambda\in\mathcal P^+$ is exactly $a_i\geq b_i$ for every $i$.  We put
\[
 A=\{a_1,\ldots,a_d\},
 \qquad
 B=\{b_1,\ldots,b_d\}.
\]
We scan the nonnegative integers from large to small.  At level $j$, we take an up-step if $j\in A\setminus B$, a down-step if $j\in B\setminus A$, and no vertical step if $j$ belongs to both sets or to neither.  After level $j$ has been scanned, the height is the difference between the numbers of elements of $A$ and $B$ that are at least $j$.  Thus the inequalities $a_i\geq b_i$ for every $i$ are equivalent to the path never going below height $0$; since $|A|=|B|$, it ends at height $0$.

Each return to height $0$ marks the end of exactly one maximal interval on which the two wings differ.  In the Young diagram this interval is one connected component of $U/I$.  Hence
\begin{equation}\label{eq:blocks-excursions}
 \kappa(\lambda)
 =\text{the number of primitive positive excursions of the path}.
\end{equation}

We choose $m$ of these excursions and, in each chosen excursion, reverse its final down-step to an up-step.  The result is a nonnegative path ending at height $2m$.  This operation is bijective: starting with a nonnegative path ending at $2m$, we successively reverse the last up-step from height $2j-1$ to height $2j$, for $j=m,m-1,\ldots,1$.  We obtain a nonnegative bridge together with $m$ distinguished primitive excursions.  The operation changes only the direction attached to a level, not that level itself, and therefore preserves the $q$-weight.

It remains to enumerate the resulting paths.  At a fixed level $j\geq0$ there are four possibilities: neither $A$ nor $B$ contains $j$, only $A$ contains $j$, only $B$ contains $j$, or both contain $j$.  We give these possibilities weights
\[
 1,
 \qquad xq^{j+1/2},
 \qquad x^{-1}q^{j+1/2},
 \qquad q^{2j+1},
\]
respectively.  The exponent of $x$ records the final height, while the $q$-weight is
\[
 q^{\sum_{a\in A}a+\sum_{b\in B}b+(|A|+|B|)/2}.
\]
For a bridge this is exactly $q^{|\lambda|}$.  By the Jacobi triple product~\cite[Chapter~2]{AndrewsPartitions},
\begin{equation}\label{eq:triple-path}
 \prod_{j\geq0}
 (1+xq^{j+1/2})(1+x^{-1}q^{j+1/2})
 =\frac{1}{(q;q)_\infty}
 \sum_{r\in\mathbb Z}x^rq^{r^2/2}.
\end{equation}
Thus unrestricted paths ending at height $2m$ have generating function
\[
 \frac{q^{2m^2}}{(q;q)_\infty}.
\]

Finally, we apply the reflection principle.  We reflect the initial segment of a path ending at height $2m$ through its first visit to height $-1$.  Up-steps and down-steps are interchanged at the same levels, so the weight is preserved, and the endpoint becomes $2m+2$.  This is a bijection from paths that fall below $0$ to unrestricted paths ending at height $2m+2$.  Therefore the nonnegative paths ending at height $2m$ have generating function
\[
 \frac{q^{2m^2}-q^{2(m+1)^2}}{(q;q)_\infty},
\]
which proves \eqref{eq:block-moment-gf}; coefficient extraction gives \eqref{eq:block-moment-coeff}.
\end{proof}

The coefficient form has the following simple interpretation.  Since
\[
 p(N)-p(N-r)
\]
counts partitions of $N$ in which the part $r$ is absent, Theorem~\ref{thm:refined-diagonal-pairing} may therefore be read as follows.

\begin{corollary}\label{cor:blocks-missing-part}
For $m\geq0$, the number of ways to choose $m$ diagonal blocks from a partition of $n$ whose successive ranks are all nonnegative equals the number of partitions of $n-2m^2$ having no part equal to $4m+2$.
\end{corollary}

In particular,
\[
 \#\{\lambda\vdash n:\lambda\in\mathcal P^+\}=p(n)-p(n-2),
\]
and
\[
 \sum_{\substack{\lambda\vdash n\\\lambda\in\mathcal P^+}}
 \kappa(\lambda)=p(n-2)-p(n-8).
\]
Taking $m=0$ recovers the classical Andrews--Bressoud enumeration: the number of partitions of $n$ whose successive ranks are all nonnegative equals the number of partitions of $n$ with no part equal to $2$, namely $p(n)-p(n-2)$~\cite{Andrews1972,Bressoud1980}.  A direct bijection was given by Corteel, Savage, and Venkatraman~\cite{CorteelSavageVenkatraman1998}; see also the more recent lattice-path treatment in~\cite{CorteelElizaldeSavage2023}.  The second displayed identity, and more generally Theorem~\ref{thm:refined-diagonal-pairing}, refine this classical case by the new diagonal-block statistic.

The binomial moments also assemble into a compact two-variable series.

\begin{corollary}\label{cor:block-distribution}
We have
\begin{equation}\label{eq:block-distribution-positive}
 \sum_{\lambda\in\mathcal P^+}t^{\kappa(\lambda)}q^{|\lambda|}
 =\frac{1+(t-2)\displaystyle\sum_{m\geq1}(t-1)^{m-1}q^{2m^2}}
 {(q;q)_\infty}.
\end{equation}
Consequently,
\begin{equation}\label{eq:block-distribution-all}
 \sum_{\lambda}z^{\kappa(\lambda)}q^{|\lambda|}
 =\frac{1+2(z-1)\displaystyle\sum_{m\geq1}(2z-1)^{m-1}q^{2m^2}}
 {(q;q)_\infty}.
\end{equation}
\end{corollary}

\begin{proof}
Using
\[
 t^k=\sum_{m\geq0}\binom{k}{m}(t-1)^m
\]
and Theorem~\ref{thm:refined-diagonal-pairing}, we obtain
\begin{align*}
 \sum_{\lambda\in\mathcal P^+}t^{\kappa(\lambda)}q^{|\lambda|}
 &=\frac{1}{(q;q)_\infty}
 \sum_{m\geq0}(t-1)^m
 \bigl(q^{2m^2}-q^{2(m+1)^2}\bigr)\\
 &=\frac{1+(t-2)\displaystyle\sum_{m\geq1}(t-1)^{m-1}q^{2m^2}}
 {(q;q)_\infty}.
\end{align*}
This proves \eqref{eq:block-distribution-positive}.  Each Boolean class indexed by $\mu\in\mathcal P^+$ contains $2^{\kappa(\mu)}$ members, all with the same number of diagonal blocks.  Hence
\[
 \sum_{\lambda}z^{\kappa(\lambda)}q^{|\lambda|}
 =\sum_{\mu\in\mathcal P^+}(2z)^{\kappa(\mu)}q^{|\mu|},
\]
and \eqref{eq:block-distribution-all} follows from \eqref{eq:block-distribution-positive} with $t=2z$.
\end{proof}

The two extreme cases clarify the role of the new pairing.  At $z=0$, only completely diagonally paired partitions survive, namely the self-conjugate partitions.  At $z=1$, the Boolean classes reassemble to give all ordinary partitions.  Thus the same diagrammatic pairing which places the self-conjugate term in Theorem~\ref{thm:rank-parity} into a pairing framework also provides a natural passage from self-conjugate partitions to the full partition function.

\section{The overpartition quotient and even principal hooks}\label{sec:overpartition-quotient}

The diagonal pairing of Section~\ref{sec:diagonal-pairing} gives a geometric extension of the self-conjugate phenomenon.  We return to the two-variable kernel of the Joint Pairing Theorem, whose signed specialization leads in a different direction.  We write
\begin{equation}\label{eq:J-kernel}
 \mathcal J(y,z;q)
 :=\sum_{\lambda}y^{r(\lambda)}z^{s(\lambda)}q^{|\lambda|}
 =\frac{1}{(yq^2;q^2)_\infty(zq;q^2)_\infty}.
\end{equation}

We recall that an overpartition is a partition in which the first occurrence of each distinct part may be overlined~\cite{CorteelLovejoy2004}.  For an overpartition $\pi$, we let $e(\pi)$ and $o(\pi)$ denote the numbers of even and odd parts, respectively, counted with multiplicity and independently of overlining.

\begin{theorem}\label{thm:pairing-quotient}
We have
\begin{equation}\label{eq:pairing-quotient}
 \sum_{\pi\in\overline{\mathcal P}}
 y^{e(\pi)}z^{o(\pi)}q^{|\pi|}
 =\frac{\mathcal J(y,z;q)}{\mathcal J(-y,-z;q)}.
\end{equation}
\end{theorem}

\begin{proof}
By \eqref{eq:J-kernel},
\begin{align*}
 \frac{\mathcal J(y,z;q)}{\mathcal J(-y,-z;q)}
 &=\frac{(-yq^2;q^2)_\infty}{(yq^2;q^2)_\infty}
   \frac{(-zq;q^2)_\infty}{(zq;q^2)_\infty}\\
 &=\prod_{j\geq1}\frac{1+yq^{2j}}{1-yq^{2j}}
   \prod_{j\geq1}\frac{1+zq^{2j-1}}{1-zq^{2j-1}}.
\end{align*}
For a fixed part size $k$, the factor
\[
 \frac{1+xq^k}{1-xq^k}
 =1+2xq^k+2x^2q^{2k}+\cdots
\]
is the local generating function for the multiplicity of the part $k$ in an overpartition: once the multiplicity is positive, there are two choices according as the first occurrence is overlined or not.  Taking $x=y$ for even $k$ and $x=z$ for odd $k$ gives \eqref{eq:pairing-quotient}.
\end{proof}

We let $\overline p(m)$ denote the number of overpartitions of $m$.  Specializing Theorem~\ref{thm:pairing-quotient} at $y=z=1$ and using Theorem~\ref{thm:rank-parity} gives a particularly simple quotient identity.

\begin{corollary}\label{cor:overpartition-pairing-quotient}
We have
\begin{equation}\label{eq:overpartition-pairing-quotient}
 \sum_{m\geq0}\overline p(m)q^m
 =\frac{\mathcal J(1,1;q)}{\mathcal J(-1,-1;q)}
 =\frac{\displaystyle\sum_{n\geq0}p(n)q^n}
        {\displaystyle\sum_{n\geq0}(-1)^n\operatorname{sc}(n)q^n}.
\end{equation}
\end{corollary}

We next give a Young-diagram realization of the same quotient.  For $m\geq0$, we let $\mathcal H_{\mathrm{even}}(2m)$ be the set of partitions of $2m$ for which the hook length of every cell on the main diagonal is even.  Equivalently, every principal hook has even length.  We let $H_d(m)$ count the members of $\mathcal H_{\mathrm{even}}(2m)$ having Durfee size $d$.

\begin{theorem}[Even principal-hook refinement]\label{thm:even-principal-hooks}
For $d\geq1$,
\begin{equation}\label{eq:Hd-gf}
 \sum_{m\geq0}H_d(m)q^m
 =\frac{2q^{\binom{d+1}{2}}(-q;q)_{d-1}}{(q;q)_d^2}.
\end{equation}
For $d=0$ the generating function is $1$.

Moreover, the right-hand side of \eqref{eq:Hd-gf} is the generating function for overpartitions whose Frobenius representation has $d$ columns.  Thus the equinumerosity between even-principal-hook partitions and overpartitions is refined by the Durfee-size/Frobenius-column statistic.
\end{theorem}

\begin{proof}
We let $\lambda$ have Durfee size $d$ and Frobenius symbol
\[
 \begin{pmatrix}
  a_1&a_2&\cdots&a_d\\
  b_1&b_2&\cdots&b_d
 \end{pmatrix},
 \qquad
 a_1>\cdots>a_d\geq0,
 \quad
 b_1>\cdots>b_d\geq0.
\]
The $i$th principal hook has length $a_i+b_i+1$.  Thus all principal hooks are even if and only if $a_i$ and $b_i$ have opposite parity for every $i$.  We write
\begin{equation}\label{eq:frob-parity-split}
 a_i=2\alpha_i+\varepsilon_i,
 \qquad
 b_i=2\beta_i+1-\varepsilon_i,
 \qquad
 \varepsilon_i\in\{0,1\}.
\end{equation}
Then
\begin{equation}\label{eq:half-size-frob}
 \frac{|\lambda|}{2}
 =d+\sum_{i=1}^d(\alpha_i+\beta_i).
\end{equation}

We fix the binary word $\varepsilon=(\varepsilon_1,\ldots,\varepsilon_d)$.  Since the top row of the Frobenius symbol is strictly decreasing, the sequence $\alpha_1,\ldots,\alpha_d$ is weakly decreasing, and equality between consecutive terms is allowed only across a transition $1\to0$ in $\varepsilon$.  Similarly, the sequence $\beta_1,\ldots,\beta_d$ is weakly decreasing, and equality is allowed only across a transition $0\to1$.

We put
\[
 D_\alpha=\{i:1\leq i<d,\ (\varepsilon_i,\varepsilon_{i+1})\neq(1,0)\},
 \qquad
 D_\beta=\{i:1\leq i<d,\ (\varepsilon_i,\varepsilon_{i+1})\neq(0,1)\}.
\]
Removing the forced descents at these positions leaves two arbitrary partitions with at most $d$ parts.  Their combined generating function is $(q;q)_d^{-2}$, while the forced descents contribute
\begin{align*}
 d+\sum_{i\in D_\alpha}i+\sum_{i\in D_\beta}i
 &=\binom{d+1}{2}
   +\sum_{\substack{1\leq i<d\\ \varepsilon_i=\varepsilon_{i+1}}}i.
\end{align*}
Hence, for fixed $\varepsilon$,
\[
 \sum q^{|\lambda|/2}
 =\frac{q^{\binom{d+1}{2}}}{(q;q)_d^2}
  q^{\sum_{\varepsilon_i=\varepsilon_{i+1}}i}.
\]
The first bit of $\varepsilon$ is arbitrary; thereafter, at each position $i$, the next bit may either agree with or differ from its predecessor.  Therefore
\begin{align*}
 \sum_{\varepsilon\in\{0,1\}^d}
 q^{\sum_{\varepsilon_i=\varepsilon_{i+1}}i}
 &=2\prod_{i=1}^{d-1}(1+q^i)\\
 &=2(-q;q)_{d-1},
\end{align*}
which proves \eqref{eq:Hd-gf}.

The same expression arises on the overpartition side.  We use the standard Frobenius representation of an overpartition~\cite{Lovejoy2005}: with $d$ columns, it consists of a top row of $d$ distinct nonnegative integers and a bottom row that is an overpartition into $d$ nonnegative parts; the weight is $d$ plus the sum of all entries.  The top row has generating function
\[
 \frac{q^{\binom d2}}{(q;q)_d},
\]
while the bottom row has generating function
\[
 \frac{(-1;q)_d}{(q;q)_d}
 =\frac{2(-q;q)_{d-1}}{(q;q)_d}.
\]
Including the contribution $q^d$ of the $d$ columns gives
\[
 q^d\frac{q^{\binom d2}}{(q;q)_d}
 \frac{2(-q;q)_{d-1}}{(q;q)_d}
 =\frac{2q^{\binom{d+1}{2}}(-q;q)_{d-1}}{(q;q)_d^2},
\]
exactly the series in \eqref{eq:Hd-gf}.  
\end{proof}

Summing over the Durfee size gives the following geometric realization of overpartitions within the pairing framework. 

\begin{corollary}\label{cor:even-hooks-overpartitions}

For every nonnegative integer $m$, the number of partitions of $2m$ whose principal hooks are all even equals the number of overpartitions of $m$.

\end{corollary}

\begin{proof}
By Theorem~\ref{thm:even-principal-hooks}, summing over $d$ counts all overpartition Frobenius representations and hence all overpartitions.
\end{proof}

Combining this with Corollary~\ref{cor:overpartition-pairing-quotient} gives the convolution announced in the introduction.

\begin{corollary}\label{cor:even-hooks-convolution}
For every $n\geq0$,
\[
 p(n)=\sum_{j=0}^n(-1)^j\operatorname{sc}(j)\,
 \#\mathcal H_{\mathrm{even}}\bigl(2(n-j)\bigr).
\]
\end{corollary}

\begin{proof}
We multiply both sides of \eqref{eq:overpartition-pairing-quotient} by
$\sum_{j\geq0}(-1)^j\operatorname{sc}(j)q^j$ and use
Corollary~\ref{cor:even-hooks-overpartitions}.
\end{proof}

\section{Simply paired partitions of negative pairing rank}

\begin{definition}
A partition is \emph{simply paired} if no part occurs more than three times.  For such a partition $\lambda$, let $b(\lambda)$ denote the number of distinct parts that participate in a pair.
\end{definition}

For a simply paired partition, we let $A$ be the set of parts that participate in a pair and $U$ the set of unpaired parts.  These sets need not be disjoint: a part in $A\cap U$ occurs three times.  Conversely, every ordered pair of finite sets $(A,U)$ of positive integers determines a unique simply paired partition, with multiplicity
\[
 2\mathbf{1}_{j\in A}+\mathbf{1}_{j\in U}
\]
for the part $j$.  Under this correspondence,
\[
 |\lambda|=2\sum_{a\in A}a+\sum_{u\in U}u,
 \qquad
 r(\lambda)=\max A,
 \qquad
 b(\lambda)=|A|,
\]
where $\max\varnothing=0$.  The statistic $s(\lambda)$ is the alternating sum of the elements of $U$ in decreasing order.

We use the following consequence of Lemma~\ref{lem:finite}:
\begin{equation}\label{eq:distinct-alt-sum}
 \sum_{\substack{U\subset\mathbb Z_{>0}\text{ finite}\\s(U)=t}}
 q^{\sum_{u\in U}u}
 =\frac{q^t}{(q^2;q^2)_t}
 \qquad(t\geq0).
\end{equation}
Indeed, this is the coefficient of $z^t$ in the $m=0$ part of the limiting form of Lemma~\ref{lem:finite}, or equivalently in
\[
 \sum_U z^{s(U)}q^{\sum_{u\in U}u}
 =\frac{1}{(zq;q^2)_\infty}.
\]

\begin{theorem}[Odd-divisor theorem]\label{thm:odd-divisor}
Among the simply paired partitions of $n$ having negative pairing rank, let $E_{\mathrm{sp}}(n)$ count those with $b(\lambda)$ even and let $O_{\mathrm{sp}}(n)$ count those with $b(\lambda)$ odd.  Then
\[
 E_{\mathrm{sp}}(n)-O_{\mathrm{sp}}(n)=d_{\mathrm{odd}}(n),
\]
where $d_{\mathrm{odd}}(n)$ is the number of odd positive divisors of $n$.
\end{theorem}

\begin{proof}
We fix the unpaired set $U$ and put $t=s(U)$.  The pairing rank is negative exactly when
\[
 r(\lambda)=\max A<t.
\]
Thus $t\geq1$ and $A$ may be any subset of $\{1,2,\ldots,t-1\}$.  Weighting such a set by $(-1)^{|A|}q^{2\sum_{a\in A}a}$ gives
\[
 \sum_{A\subseteq\{1,\ldots,t-1\}}
 (-1)^{|A|}q^{2\sum_{a\in A}a}
 =\prod_{j=1}^{t-1}(1-q^{2j})
 =(q^2;q^2)_{t-1}.
\]
Combining this with \eqref{eq:distinct-alt-sum}, we obtain
\begin{align*}
 \sum_{n\geq1}\bigl(E_{\mathrm{sp}}(n)-O_{\mathrm{sp}}(n)\bigr)q^n
 &=\sum_{t\geq1}
 \frac{q^t}{(q^2;q^2)_t}(q^2;q^2)_{t-1}\\
 &=\sum_{t\geq1}\frac{q^t}{1-q^{2t}}\\
 &=\sum_{t\geq1}\sum_{j\geq0}q^{(2j+1)t}.
\end{align*}
The coefficient of $q^n$ in the last series counts factorizations
\[
 n=(2j+1)t.
\]
These are in bijection with the odd divisors $2j+1$ of $n$.  Thus that coefficient is $d_{\mathrm{odd}}(n)$, proving the theorem.
\end{proof}

We let $\overline p_{\mathrm{odd}}(n)$ denote the number of overpartitions of $n$ into odd parts.

\begin{theorem}[Odd-overpartition theorem]\label{thm:odd-overpartition}
Let $M(n)$ be the number of simply paired partitions of $n$ having negative pairing rank.  Then, for every $n\geq1$,
\[
 M(n)=\frac12\,\overline p_{\mathrm{odd}}(n).
\]
\end{theorem}

\begin{proof}
As in the preceding proof, after fixing $t=s(U)\geq1$, the set $A$ may be any subset of $\{1,2,\ldots,t-1\}$.  If we omit the sign $(-1)^{|A|}$, its generating function is
\[
 \sum_{A\subseteq\{1,\ldots,t-1\}}
 q^{2\sum_{a\in A}a}
 =\prod_{j=1}^{t-1}(1+q^{2j})
 =(-q^2;q^2)_{t-1}.
\]
It follows from \eqref{eq:distinct-alt-sum} that
\begin{equation}\label{eq:M-series}
 \sum_{n\geq1}M(n)q^n
 =\sum_{t\geq1}
 \frac{q^t(-q^2;q^2)_{t-1}}{(q^2;q^2)_t}.
\end{equation}

We apply the $q$-binomial theorem~\cite[Chapter~2]{AndrewsPartitions} in the form
\[
 \frac{(az;Q)_\infty}{(z;Q)_\infty}
 =\sum_{t\geq0}\frac{(a;Q)_t}{(Q;Q)_t}z^t.
\]
Taking $Q=q^2$, $a=-1$, and $z=q$ gives
\[
 \frac{(-q;q^2)_\infty}{(q;q^2)_\infty}
 =\sum_{t\geq0}\frac{(-1;q^2)_t}{(q^2;q^2)_t}q^t.
\]
For $t\geq1$,
\[
 (-1;q^2)_t
 =2(-q^2;q^2)_{t-1}.
\]
Thus \eqref{eq:M-series} becomes
\begin{align*}
 \sum_{n\geq1}M(n)q^n
 &=\frac12\left(
 \frac{(-q;q^2)_\infty}{(q;q^2)_\infty}-1
 \right).
\end{align*}
Finally,
\[
 \sum_{n\geq0}\overline p_{\mathrm{odd}}(n)q^n
 =\prod_{j\geq1}\frac{1+q^{2j-1}}{1-q^{2j-1}}
 =\frac{(-q;q^2)_\infty}{(q;q^2)_\infty}.
\]
Since $\overline p_{\mathrm{odd}}(0)=1$, comparison of coefficients of $q^n$ for $n\geq1$ yields
\[
 M(n)=\frac12\,\overline p_{\mathrm{odd}}(n),
\]
as claimed.
\end{proof}

\section{A square-alternating theorem for pairing rank minus two}

For $d\geq0$, we let $A_d(n)$ denote the number of simply paired partitions of $n$ having pairing rank $-d$, and write
\[
 A_d(q)=\sum_{n\geq0}A_d(n)q^n.
\]
We set $A_d(n)=0$ when $n<0$.

To derive $A_d(q)$, we fix the largest paired part $r=\max A$.  If $r=0$, then $A$ is empty and the unpaired set has alternating sum $d$.  If $r\geq1$, then $A$ contains $r$ and may contain an arbitrary subset of $\{1,\ldots,r-1\}$.  Using \eqref{eq:distinct-alt-sum}, we therefore obtain
\begin{equation}\label{eq:Ad-series}
 A_d(q)
 =\frac{q^d}{(q^2;q^2)_d}
 +\sum_{r\geq1}
 \frac{q^{3r+d}(-q^2;q^2)_{r-1}}
 {(q^2;q^2)_{r+d}}.
\end{equation}

We first record a recurrence between consecutive rank layers.

\begin{lemma}\label{lem:Ad-recurrence}
For every $d\geq0$,
\begin{align}
 (1+q^{2d+2})A_{d+1}(q)
 ={}&(q+q^{2d})A_d(q)
 -\frac{q^{3d}}{(q^2;q^2)_d}
 +\frac{q^{3d+3}}{(q^2;q^2)_{d+1}}.
 \label{eq:Ad-recurrence}
\end{align}
\end{lemma}

\begin{proof}
We put $Q=q^2$ and define
\[
 a_0=1,
 \qquad
 a_r=Q^r(-Q;Q)_{r-1}\quad(r\geq1),
 \qquad
 b_t=\frac{q^t}{(Q;Q)_t}.
\]
Equation \eqref{eq:Ad-series} may then be written as
\[
 A_d(q)=\sum_{r\geq0}a_rb_{r+d}.
\]
The relation
\[
 b_{t+1}(1-Q^{t+1})=qb_t
\]
gives
\begin{equation}\label{eq:Ad-first-step}
 A_{d+1}(q)-Q^{d+1}S_d(q)=qA_d(q),
 \qquad
 S_d(q)=\sum_{r\geq0}Q^ra_rb_{r+d+1}.
\end{equation}
For $r\geq1$ we have
\[
 Q^ra_r=\frac{a_{r+1}}{Q}-a_r,
\]
while the term with $r=0$ must be kept separately.  Consequently,
\begin{align*}
 S_d(q)
 & =b_{d+1}
 +\frac1Q\sum_{r\geq1}a_{r+1}b_{r+d+1}
 -\sum_{r\geq1}a_rb_{r+d+1}\\
 & =b_{d+1}
 +\frac1Q\bigl(A_d-b_d-a_1b_{d+1}\bigr)
 -\bigl(A_{d+1}-b_{d+1}\bigr).
\end{align*}
Since $a_1=Q$, this simplifies to
\[
 S_d(q)=\frac{A_d(q)-b_d}{Q}-A_{d+1}(q)+b_{d+1}.
\]
Substitution into \eqref{eq:Ad-first-step} gives
\[
 (1+Q^{d+1})A_{d+1}(q)
 =(q+Q^d)A_d(q)-Q^db_d+Q^{d+1}b_{d+1}.
\]
Replacing $Q$ and the $b_t$ by their definitions yields \eqref{eq:Ad-recurrence}.
\end{proof}

We let
\[
 M(q)=\sum_{n\geq1}M(n)q^n,
\]
where $M(n)$ is the total number of simply paired partitions of $n$ having negative pairing rank, as in Theorem~\ref{thm:odd-overpartition}.

\begin{lemma}\label{lem:A2-M}
The generating functions $A_2(q)$ and $M(q)$ satisfy
\begin{equation}\label{eq:A2-M}
 (1+q^2)(1+q^4)A_2(q)
 =q(1-q^2)M(q)-\frac{q^3}{1+q}
 +\frac{q^6}{(1-q^2)^2}.
\end{equation}
\end{lemma}

\begin{proof}
Taking $d=0$ and $d=1$ in Lemma~\ref{lem:Ad-recurrence} gives
\begin{align}
 (1+q^2)A_1(q)
 &=(1+q)A_0(q)-1+\frac{q^3}{1-q^2},
 \label{eq:A1-A0}\\
 (1+q^4)A_2(q)
 &=(q+q^2)A_1(q)-\frac{q^3}{1-q^2}
 +\frac{q^6}{(1-q^2)(1-q^4)}.
 \label{eq:A2-A1}
\end{align}
Multiplying \eqref{eq:A2-A1} by $1+q^2$ and substituting \eqref{eq:A1-A0}, we find
\begin{align}
 &(1+q^2)(1+q^4)A_2(q)\notag\\
 &\quad=(q+q^2)
 \left((1+q)A_0(q)-1+\frac{q^3}{1-q^2}\right)
 -\frac{q^3(1+q^2)}{1-q^2}
 +\frac{q^6(1+q^2)}{(1-q^2)(1-q^4)}\notag\\
 &\quad=q(1+q)^2A_0(q)-q(1+q)
 -\frac{q^3}{1+q}+\frac{q^6}{(1-q^2)^2}.
 \label{eq:A2-A0}
\end{align}

To express $A_0(q)$ in terms of $M(q)$, we use \eqref{eq:Ad-series}:
\[
 A_0(q)=1+\sum_{r\geq1}
 \frac{q^{3r}(-q^2;q^2)_{r-1}}{(q^2;q^2)_r}.
\]
Since
\[
 (-1;q^2)_r=2(-q^2;q^2)_{r-1}\qquad(r\geq1),
\]
the $q$-binomial theorem gives
\begin{align*}
 A_0(q)
 &=\frac12\left(
 1+\frac{(-q^3;q^2)_\infty}{(q^3;q^2)_\infty}
 \right).
\end{align*}
On the other hand, Theorem~\ref{thm:odd-overpartition} gives
\[
 1+2M(q)=\frac{(-q;q^2)_\infty}{(q;q^2)_\infty}.
\]
Removing the factors corresponding to the part $1$ gives
\[
 \frac{(-q^3;q^2)_\infty}{(q^3;q^2)_\infty}
 =\frac{1-q}{1+q}\bigl(1+2M(q)\bigr).
\]
It follows that
\[
 A_0(q)=\frac{1+(1-q)M(q)}{1+q}.
\]
Substitution into \eqref{eq:A2-A0} proves \eqref{eq:A2-M}.
\end{proof}

For a nonnegative integer $m$, we define
\[
 \chi(m)=
 \begin{cases}
 1,&m=2^\alpha u^2\text{ for some }\alpha\geq0
       \text{ and some positive odd integer }u,\\
 0,&\text{otherwise}.
 \end{cases}
\]
Thus $\chi(0)=0$, and for $m\geq1$ the value $\chi(m)$ is $1$ precisely when the odd part of $m$ is a square.

\begin{theorem}[Square-switch theorem]\label{thm:square-switch}
For every $n\geq0$,
\begin{equation}\label{eq:square-switch}
 A_2(n+8)+A_2(n)
 \equiv n+1+\chi(n+3)+\chi(n+7)\pmod2.
\end{equation}
In particular, if $n$ is even, then $A_2(n+8)$ and $A_2(n)$ have the same parity if and only if one of $n+3$ and $n+7$ is a perfect square.
\end{theorem}

\begin{proof}
We first determine the parity of $M(n)$.  By Theorem~\ref{thm:odd-overpartition},
\begin{align*}
 1+2M(q)
 &=\prod_{j\geq1}
 \frac{1+q^{2j-1}}{1-q^{2j-1}}\\
 &=\prod_{j\geq1}
 \left(1+\frac{2q^{2j-1}}{1-q^{2j-1}}\right).
\end{align*}
Modulo $4$, every product containing at least two nonconstant factors vanishes.  Hence
\[
 1+2M(q)
 \equiv1+2\sum_{j\geq1}
 \frac{q^{2j-1}}{1-q^{2j-1}}\pmod4.
\]
Subtracting $1$ and dividing coefficientwise by $2$ gives
\begin{equation}\label{eq:M-parity}
 M(q)\equiv\sum_{j\geq1}
 \frac{q^{2j-1}}{1-q^{2j-1}}\pmod2.
\end{equation}
The coefficient of $q^m$ on the right of \eqref{eq:M-parity} is the number $d_{\mathrm{odd}}(m)$ of odd divisors of $m$.  If $m=2^\alpha v$ with $v$ odd, then
\[
 d_{\mathrm{odd}}(m)=d(v),
\]
which is odd precisely when $v$ is a square.  Therefore
\begin{equation}\label{eq:M-chi}
 M(m)\equiv\chi(m)\pmod2
 \qquad(m\geq0),
\end{equation}
where $M(0)=0$.

We define
\[
 \Delta_2(m)=A_2(m)+A_2(m-2)+A_2(m-4)+A_2(m-6).
\]
Modulo $2$, the two elementary series in \eqref{eq:A2-M} satisfy
\[
 -\frac{q^3}{1+q}+\frac{q^6}{(1-q^2)^2}
 \equiv\sum_{m\geq3}q^m+\sum_{j\geq0}q^{4j+6}\pmod2.
\]
We let $\eta(m)$ denote the coefficient of $q^m$ in the series on the right.  For $m\geq6$,
\begin{equation}\label{eq:eta}
 \eta(m)=
 \begin{cases}
 0,&m\equiv2\pmod4,\\
 1,&m\not\equiv2\pmod4.
 \end{cases}
\end{equation}
Taking coefficients in \eqref{eq:A2-M} and using \eqref{eq:M-chi}, we obtain
\begin{equation}\label{eq:Delta-parity}
 \Delta_2(m)
 \equiv\chi(m-1)+\chi(m-3)+\eta(m)\pmod2.
\end{equation}

The intermediate terms cancel in pairs modulo $2$, so
\begin{align*}
 A_2(n+8)+A_2(n)
 &\equiv\Delta_2(n+8)+\Delta_2(n+6)\\
 &\equiv\chi(n+7)+\chi(n+3)
   +\eta(n+8)+\eta(n+6)\pmod2.
\end{align*}
If $n$ is even, exactly one of $n+6$ and $n+8$ is congruent to $2$ modulo $4$, and \eqref{eq:eta} gives
\[
 \eta(n+8)+\eta(n+6)\equiv1\pmod2.
\]
If $n$ is odd, both $n+6$ and $n+8$ are odd, so both corresponding values of $\eta$ are $1$ and their sum is $0$ modulo $2$.  Thus, in all cases,
\[
 \eta(n+8)+\eta(n+6)\equiv n+1\pmod2.
\]
This proves \eqref{eq:square-switch}.

When $n$ is even, the integers $n+3$ and $n+7$ are odd, so $\chi$ is their ordinary perfect-square indicator.  Moreover, according to the residue of $n$ modulo $8$, at most one of these integers can be congruent to $1$ modulo $8$; hence they cannot both be squares.  It follows from \eqref{eq:square-switch} that the two values of $A_2$ have the same parity exactly when one of $n+3$ and $n+7$ is a square.
\end{proof}

The square-switch theorem immediately yields four fixed parity progressions.

\begin{corollary}\label{cor:A2-progressions}
For every $n\geq0$,
\begin{align*}
 A_2(16n)&\equiv0\pmod2,
 &A_2(16n+4)&\equiv1\pmod2,\\
 A_2(16n+8)&\equiv1\pmod2,
 &A_2(16n+12)&\equiv0\pmod2.
\end{align*}
\end{corollary}

\begin{proof}
If $m\equiv0$ or $4\pmod8$, then neither $m+3$ nor $m+7$ is congruent to $1$ modulo $8$, and hence neither is an odd square.  Theorem~\ref{thm:square-switch} therefore gives
\[
 A_2(m+8)\equiv A_2(m)+1\pmod2.
\]
Finally, $A_2(0)=0$, while $A_2(4)=1$; indeed, $(3,1)$ is the unique simply paired partition of $4$ having pairing rank $-2$.  Induction along the two progressions $m\equiv0\pmod8$ and $m\equiv4\pmod8$ proves all four congruences.
\end{proof}

\section{A Toeplitz determinant for the pairing rank}

For $m\in\mathbb Z$, we let
\[
 R_m(q)=\sum_{\substack{\lambda\\\rho(\lambda)=m}}q^{|\lambda|}
\]
be the generating function for partitions of pairing rank $m$.  The limiting
form of Lemma~\ref{lem:finite}, with the variables that mark $r(\lambda)$ and
$s(\lambda)$ specialized to $z$ and $z^{-1}$, respectively, gives
\begin{equation}\label{eq:rank-bivariate-gf}
 \mathcal R(z;q)
 :=\sum_{m\in\mathbb Z}R_m(q)z^m
 =\frac{1}{(zq^2;q^2)_\infty(z^{-1}q;q^2)_\infty}.
\end{equation}
We first derive a finite identity and then take its coefficientwise limit.  We use standard notation for complete homogeneous and Schur symmetric functions, as in Macdonald~\cite[Chapter~I]{Macdonald1995}.

We put
\[
 x=(q^2,q^4,q^6,\ldots),
 \qquad
 y=(q,q^3,q^5,\ldots).
\]
For an alphabet $x=(x_1,x_2,\ldots)$, we write
\[
 H(x;z)=\prod_{r\geq1}\frac{1}{1-x_rz}
       =\sum_{a\geq0}h_a(x)z^a,
\]
where $h_a(x)$ is the complete homogeneous symmetric function, with the
convention $h_a(x)=0$ for $a<0$.  Equation
\eqref{eq:rank-bivariate-gf} becomes
\begin{equation}\label{eq:rank-H-factorization}
 \mathcal R(z;q)=H(x;z)H(y;z^{-1}).
\end{equation}
Expanding both factors and extracting the coefficient of $z^{i-j}$ gives
\begin{equation}\label{eq:rank-convolution}
 R_{i-j}(q)
 =\sum_{r\in\mathbb Z}h_{r+i}(x)h_{r+j}(y).
\end{equation}
Indeed, in a term $h_a(x)h_b(y)z^{a-b}$, the substitutions
$a=r+i$ and $b=r+j$ give $a-b=i-j$.

\begin{theorem}[Finite determinant identity]\label{thm:finite-rank-determinant}
For every positive integer $k$,
\begin{equation}\label{eq:finite-rank-determinant}
 \det\bigl(R_{i-j}(q)\bigr)_{1\leq i,j\leq k}
 =\sum_{\ell(\lambda)\leq k}
 s_\lambda(q^2,q^4,q^6,\ldots)
 s_\lambda(q,q^3,q^5,\ldots).
\end{equation}
\end{theorem}

\begin{proof}
Equation~\eqref{eq:rank-convolution} expresses the matrix on the left of
\eqref{eq:finite-rank-determinant} as the product of the two infinite matrices
\[
 \bigl(h_{r+i}(x)\bigr)_{\substack{1\leq i\leq k\\r\in\mathbb Z}}
 \quad\text{and}\quad
 \bigl(h_{r+j}(y)\bigr)_{\substack{r\in\mathbb Z\\1\leq j\leq k}}.
\]
Applying Cauchy--Binet coefficientwise therefore gives
\begin{align}
 \det\bigl(R_{i-j}(q)\bigr)_{1\leq i,j\leq k}
 ={}&\sum_{r_1>r_2>\cdots>r_k}
 \det\bigl(h_{r_j+i}(x)\bigr)_{1\leq i,j\leq k}
 \det\bigl(h_{r_i+j}(y)\bigr)_{1\leq i,j\leq k}.
 \label{eq:rank-cauchy-binet}
\end{align}
We first restrict the
indices to $-k\leq r\leq L$ and apply the ordinary finite Cauchy--Binet
formula.  Indices $r<-k$ give zero, while the $q$-order of the terms tends to
infinity with $r$.  Hence every coefficient stabilizes as $L\to\infty$.
Writing the selected indices in decreasing rather than increasing order
reverses the columns of the first determinant and the rows of the second, so
the two signs cancel.

We set
\[
 \lambda_j=r_j+j,
 \qquad 1\leq j\leq k.
\]
Because $r_j>r_{j+1}$, we have
\[
 \lambda_j-\lambda_{j+1}=r_j-r_{j+1}-1\geq0.
\]
If $\lambda_k<0$, then the $k$th column of the first determinant in
\eqref{eq:rank-cauchy-binet} is zero.  Thus the nonzero terms are indexed
precisely by partitions $\lambda=(\lambda_1,\ldots,\lambda_k)$ with at most
$k$ parts; conversely, such a partition determines the indices uniquely by
$r_j=\lambda_j-j$.

The Jacobi--Trudi identity now gives
\begin{align*}
 \det\bigl(h_{r_j+i}(x)\bigr)_{1\leq i,j\leq k}
 &=\det\bigl(h_{\lambda_j-j+i}(x)\bigr)_{1\leq i,j\leq k}
 =s_\lambda(x),\\
 \det\bigl(h_{r_i+j}(y)\bigr)_{1\leq i,j\leq k}
 &=\det\bigl(h_{\lambda_i-i+j}(y)\bigr)_{1\leq i,j\leq k}
 =s_\lambda(y).
\end{align*}
The first determinant is the transpose of the usual
Jacobi--Trudi determinant.  Substitution in
\eqref{eq:rank-cauchy-binet} proves \eqref{eq:finite-rank-determinant}.
\end{proof}

\begin{theorem}[Limiting determinant identity]\label{thm:limiting-rank-determinant}
Coefficientwise as a formal power series in $q$,
\begin{equation}\label{eq:limiting-rank-determinant}
 \lim_{k\to\infty}
 \det\bigl(R_{i-j}(q)\bigr)_{1\leq i,j\leq k}
 =\prod_{r\geq1}\frac{1}{(1-q^{2r+1})^r}.
\end{equation}
The identity also holds analytically for $|q|<1$.
\end{theorem}

\begin{proof}
Every monomial occurring in
$s_\lambda(q^2,q^4,q^6,\ldots)$ has $q$-degree at least $2|\lambda|$, while
every monomial occurring in $s_\lambda(q,q^3,q^5,\ldots)$ has $q$-degree at
least $|\lambda|$.  Their product therefore has $q$-degree at least
$3|\lambda|$.  For a fixed power $q^N$, only partitions with
$|\lambda|\leq N/3$ can contribute.  It follows from
Theorem~\ref{thm:finite-rank-determinant} that the coefficientwise limit exists
and equals
\[
 \sum_\lambda s_\lambda(x)s_\lambda(y).
\]
By the Cauchy identity for Schur functions,
\[
 \sum_\lambda s_\lambda(x)s_\lambda(y)
 =\prod_{i,j\geq1}\frac{1}{1-x_iy_j}
 =\prod_{i,j\geq1}\frac{1}{1-q^{2i+2j-1}}.
\]
For a fixed $r\geq1$, there are exactly $r$ positive integer pairs $(i,j)$
satisfying $i+j-1=r$.  Collecting the corresponding factors gives
\[
 \prod_{i,j\geq1}\frac{1}{1-q^{2i+2j-1}}
 =\prod_{r\geq1}\frac{1}{(1-q^{2r+1})^r},
\]
which proves \eqref{eq:limiting-rank-determinant} formally.

For $|q|<1$, the Schur-function sum converges absolutely, since its absolute
value is bounded termwise by the same sum with $q$ replaced by $|q|$, whose
Cauchy product is
\[
 \prod_{i,j\geq1}\frac{1}{1-|q|^{2i+2j-1}}<\infty.
\]
The finite sums in \eqref{eq:finite-rank-determinant} therefore converge to the
same product analytically.
\end{proof}

\begin{remark}
MacMahon's classical generating function for plane partitions is
\[
 \prod_{r\geq1}\frac{1}{(1-q^r)^r}
 \qquad\text{\cite{MacMahon1916}}.
\]
Thus the product in \eqref{eq:limiting-rank-determinant} retains the characteristic linear multiplicity $r$ of MacMahon's product, but places the factor of multiplicity $r$ at the odd exponent $2r+1$.  This is the precise sense in which the limiting Toeplitz determinant is related to the plane-partition product.
\end{remark}

\end{document}